\documentclass[11pt]{article}
\usepackage{amssymb, amsmath, amsthm, graphicx}
\usepackage[left=1in,top=1in,right=1in]{geometry}
\date{}
\allowdisplaybreaks

\theoremstyle{definition}
\newtheorem{theorem}{Theorem}[section]
\newtheorem{lemma}[theorem]{Lemma}
\newtheorem{definition}[theorem]{Definition}
\newtheorem{remark}[theorem]{Remark}
\newtheorem{corollary}[theorem]{Corollary}

\newtheorem{fact}[theorem]{Fact}
\counterwithin{equation}{section}

\title{A positive fraction Erd\H os-Szekeres theorem and its applications}
\author{Andrew Suk\thanks{Department of Mathematics, University of California at San Diego, La Jolla, CA, 92093 USA. Supported by NSF CAREER award DMS-1800746,  NSF award DMS-1952786, and an Alfred Sloan Fellowship. Email: {\tt asuk@ucsd.edu}.}\and Ji Zeng\thanks{Department of Mathematics, University of California at San Diego, La Jolla, CA, 92093 USA. Supported by NSF grant DMS-1800746. Email:
{\tt jzeng@ucsd.edu}.}}

\begin{document}

\maketitle

\begin{abstract}
A famous theorem of Erd\H os and Szekeres states that any sequence of $n$ distinct real numbers contains a monotone subsequence of length at least $\sqrt{n}$. Here, we prove a positive fraction version of this theorem.  For $n > (k-1)^2$, any sequence $A$ of $n$ distinct real numbers contains a collection of subsets $A_1,\ldots, A_k \subset A$, appearing sequentially, all of size $s=\Omega(n/k^2)$, such that every subsequence $(a_1,\ldots, a_k)$, with $a_i \in A_i$, is increasing, or every such subsequence is decreasing. The subsequence $S = (A_1,\ldots, A_k)$ described above is called \emph{block-monotone of depth $k$ and block-size $s$}.  Our theorem is asymptotically best possible and follows from a more general Ramsey-type result for monotone paths, which we find of independent interest. We also show that for any positive integer $k$, any finite sequence of distinct real numbers can be partitioned into $O(k^2\log k)$ block-monotone subsequences of depth at least $k$, upon deleting at most $(k-1)^2$ entries. We apply our results to mutually avoiding planar point sets and biarc diagrams in graph drawing.
\end{abstract}

\section{Introduction}\label{sec:intro}

In 1935, Erd\H os and Szekeres \cite{es35} proved that any sequence of $n$ distinct real numbers contains a monotone subsequence of length at least $\sqrt{n}$.  This is a classical result in combinatorics and its generalizations and extensions have many important consequences in geometry, probability, and computer science. See Steele \cite{St} for 7 different proofs along with several applications.

In this paper, we prove a positive fraction version of the Erd\H os-Szekeres theorem. We state this theorem using the following notion: A sequence $(a_1, a_2,\dots,a_{ks})$ of $ks$ distinct real numbers is said to be \emph{block-increasing} (resp. \emph{block-decreasing}) with \emph{depth} $k$ and \emph{block-size} $s$ if every subsequence $(a_{i_1},a_{i_2},\dots, a_{i_k})$, for $(j-1)s<i_j\leq js$, is increasing (resp. decreasing). We call a sequence \emph{block-monotone} if it's either block-increasing or block-decreasing. In other words, such a sequence consists of $k$ parts, each of size $s$ and appearing sequentially, such that the parts (i.e. blocks) are in a monotone position.

\begin{theorem}\label{main}
Let $k$ and $n > (k-1)^2$ be positive integers. Then every sequence of $n$ distinct real numbers contains a block-monotone subsequence of depth $k$ and block-size $s = \Omega(n/k^2)$. Furthermore, such a subsequence can be computed within $O(n^2\log n)$ time.
\end{theorem}
\noindent We prove Theorem \ref{main} by establishing a more general Ramsey-type result for monotone paths, which we describe in detail in the next section. The theorem is also asymptotically best possible, see Remark~\ref{extreme_construction_remark}.

By a repeated application of Theorem \ref{main}, we can decompose any sequence of $n$ distinct real numbers into $O(k\log n)$ block-monotone subsequences of depth $k$ upon deleting at most $(k-1)^2$ entries. Our next result shows that we can obtain such a partition, where the number of parts doesn't depend on $n$.

\begin{theorem}\label{partition}
For any positive integer $k$, every finite sequence of distinct real numbers can be partitioned into at most $O(k^2\log k)$ block-monotone subsequences of depth at least $k$ upon deleting at most $(k-1)^2$ entries.
\end{theorem}
\noindent Our proof of Theorem~\ref{partition} is constructive and implies an algorithm for the claimed partition whose time complexity is polynomial in $k$ and $n$, where $n$ is the length of the given sequence.

Theorem~\ref{partition} is inspired by a similar problem of partitioning planar point sets into convex-positioned clusters studied by P{\'o}r and Valtr~\cite{por2002partitioned}. A positive fraction Erd\H os-Szekeres-type result for convex polygons is given previously by B{\'a}r{\'a}ny and Valtr \cite{barany1998positive}.

We give two applications of Theorems~\ref{main}~and~\ref{partition}.

\medskip

\noindent \textbf{Mutually avoiding sets.}  Let $A$ and $B$ be finite point sets of $\mathbb{R}^2$ in \emph{general position}, that is, no three points are collinear. We say that $A$ and $B$ are \emph{mutually avoiding} if no line generated by a pair of points in $A$ intersects the convex hull of $B$, and vice versa.  Aronov et al.~\cite{aronov1991crossing} used the Erd\H os-Szekeres Theorem to show that every $n$-element planar point set $P$ in general position contains subsets $A,B\subset P$, each of size $\Omega(\sqrt{n})$, s.t. $A$ and $B$ are mutually avoiding. Valtr \cite{valtr1997mutually} showed that this bound is asymptotically best possible by slightly perturbing the points in an $\sqrt{n}\times \sqrt{n}$ grid.  Following the same ideas of Aronov et al., we can use Theorem~\ref{main} to obtain the following.

\begin{theorem}
\label{avoiding}For every positive integer $k$ there is a constant $\epsilon_k=\Omega(\frac{1}{k^2})$ s.t. every sufficiently large point set $P$ in the plane in general position contains $2k$ disjoint subsets $A_1,\dots,A_k,B_1,\dots,B_k$, each of size at least $\epsilon_k|P|$, s.t. every pair of sets $A=\{a_1,\dots,a_k\}$ and $B=\{b_1,\dots,b_k\}$, with $a_i\in A_i$ and $b_i\in B_i$, are mutually avoiding.
\end{theorem}

\noindent This improves an earlier result of Mirzaei and the first author \cite{mirzaei2020positive}, who proved the theorem above with $\epsilon_k=\Omega(\frac{1}{k^4})$. The result above is asymptotically best possible for both $k$ and $|P|$: Consider a $k\times k$ grid $G$ and replace each point with a cluster of $|P|/k^2$ points placed very close to each other so that the resulting point set $P$ is in general position. If we can find subsets $A_i$'s and $B_i$'s as in Theorem~\ref{avoiding}, but each of size $\epsilon'_k |P|$ with $\epsilon'_k=\omega(\frac{1}{k^2})$, then we can find mutually avoiding subsets in $G$ of size $\omega(k)$, contradicting Valtr's result in \cite{valtr1997mutually}.

Finally, let us remark that a recent result due to Pach, Rubin, and Tardos \cite{pach2021planar} shows that every $n$-element planar point set in general position determines at least $n/e^{O(\sqrt{\log n})}$ pairwise crossing segments.  By using Theorem~\ref{avoiding} instead of Lemma~3.3 from their paper, one can improve the constant hidden in the $O$-notation.

\medskip

\noindent\textbf{Monotone biarc diagrams.} A \emph{proper arc diagram} is a drawing of a graph in the plane, whose vertices are points placed on the $x$-axis, called the \emph{spine}, and each edge is drawn as a half-circle.  A classic result of Bernhard and Kainen \cite{BH} shows that a planar graph admits a \emph{planar} proper arc diagram if and only if it's a subgraph of a planar Hamiltonian graph. A \emph{monotone biarc diagram} is a drawing of a graph in the plane, whose vertices are placed on a spine, and each edge is drawn either as a half-circle or two half-circles centered on the spine, forming a continuous $x$-monotone biarc. See Figure~\ref{fig:biarc_example} for an illustration. In \cite{di2005curve}, Di Giacomo et al. showed that every planar graph can be drawn as a \emph{planar} monotone biarc diagram.

Using the Erd\H os-Szekeres Theorem, Bar-Yehuda and Fogel \cite{yehuda1998partitioning} showed that every graph $G=(V,E)$, with a given order on $V$, has a \emph{double-paged book embedding} with at most $O(\sqrt{E})$ pages. That is, $E$ can be partitioned into $O(\sqrt{|E|})$ parts, s.t. for each part $E_i$, $(V,E_i)$ can be drawn as a planar monotone biarc diagram, and $V$ appears on the spine with the given order. Our next result shows that we can significantly reduce the number of pages (parts), if we allow a small fraction of the pairs of edges to cross on each page.  

\begin{theorem}\label{biarc}
For any $\epsilon>0$ and a graph $G=(V,E)$, where $V$ is an ordered set, $E$ can be partitioned into $O(\epsilon^{-2}\log(\epsilon^{-1})\log(|E|))$ subsets $E_i$ s.t. each $(V,E_i)$ can be drawn as a monotone biarc diagram having no more than $\epsilon|E_i|^2$ crossing edge-pairs, and $V$ appears on the spine with the given order.
\end{theorem}

This paper is organized as follows: In Section~\ref{sec:main_proof}, we prove Theorem~\ref{main} in the setting of monotone paths in multicolored ordered graphs. Section~\ref{sec:partition} is devoted to the proof of Theorem~\ref{partition}. In Section~\ref{sec:applications}, we present proofs for the applications claimed above. Section~\ref{sec:remarks} lists some remarks.

\section{A positive fraction result for monotone paths}\label{sec:main_proof}

Several authors \cite{FPSS,MS,MSW} observed that the Erd\H os-Szekeres theorem generalizes to the following graph-theoretic setting. Let $G$ be a graph with vertex set $[n] = \{1,\ldots ,n\}$. A \emph{monotone path of length} $k$ in $G$ is a $k$-tuple $(v_1,\dots, v_k)$ of vertices s.t. $v_i<v_j$ for all $i<j$ and all edges $v_iv_{i + 1}$, for $i \in [k-1]$, are in $G$.

\begin{theorem}\label{espath} Let $\chi$ be a $q$-coloring of the pairs of $[n]$.  Then there must be a monochromatic monotone path of length at least $n^{1/q}$. 
\end{theorem}

Given subsets $A,B\subset [n]$, we write $A < B$ if every element in $A$ is less than every element in $B$.   

\begin{definition}
Let $G$ be a graph with vertex set $[n]$ and let $V_1,\ldots, V_k \subset [n]$ and $p_1,\ldots, p_{k + 1} \in [n]$. Then we say that $(p_1,V_1,p_2,V_2,p_3,\ldots,p_{k}, V_k,p_{k + 1})$ is a \emph{block-monotone path of depth $k$ and block-size $s$} if 
\begin{enumerate}
    \item  $|V_i| = s$ for all $i$, 
    \item  we have $p_1 < V_1 < p_2 < V_2 < p_3 < \ldots < p_k < V_k < p_{k + 1},$
    \item and every $(2k + 1)$-tuple of the form
        \[(p_1,v_1,p_2,v_2,\ldots, p_k,v_k,p_{k + 1}),\]
    where $v_i \in V_i$, is a monotone path in $G$.
\end{enumerate}
\end{definition}

\noindent Our main result in this section is the following Ramsey-type theorem.

\begin{theorem}\label{path}
There is an absolute constant $c>0$ s.t. the following holds. Given integers $q\geq 2$, $k \geq 1$, and $n \geq (ck)^q$, let $\chi$ be a $q$-coloring of the pairs of $[n]$.  Then $\chi$ produces a monochromatic block-monotone path of depth $k$ and block-size $s \geq \frac{n}{(ck)^q}$.
  
\end{theorem}

\noindent A careful calculation shows that we can take $c = 40$ in the theorem above.  We will need the following lemma.

\begin{lemma}\label{2path}
Let $q\geq 2$ and $N >  3^{q}$.  Then for any $q$-coloring of the pairs of $[N]$, there is a monochromatic block-monotone path of depth $1$ and block-size $s\geq \frac{N}{q3^{3q}}$.
\end{lemma}

\begin{proof}
Let $\chi$ be a $q$-coloring of the pairs of $[N]$, and set $r = 3^q$.  By Theorem \ref{espath}, every subset of size $r$ of $[N]$ gives rise to a monochromatic monotone path of length 3.  Hence, $\chi$ produces at least \[\frac{\binom{N}{r}}{\binom{N-3}{r-3}} \geq \frac{6}{r^3}\binom{N}{3}\] monochromatic monotone paths of length $3$ in $[N]$.  Hence, there are at least $\frac{6}{qr^3}\binom{N}{3}$ monochromatic monotone paths of length 3, all of which have the same color. By averaging, there are two vertices $p_1,p_2 \in [N]$, s.t. at least $\frac{N}{qr^3}$ of these monochromatic monotone paths of length 3 start at vertex $p_1$ and ends at vertex $p_2$.  By setting $V_1$ to be the ``middle'' vertices of these paths, $(p_1,V_1,p_2)$ is a monochromatic block-monotone path of depth $1$ and block-size $s\geq \frac{N}{qr^3} = \frac{N}{q3^{3q}}$.\end{proof}

\begin{proof}[Proof of Theorem \ref{path}]

Let $\chi$ be a $q$-coloring of the pairs of $[n]$ and let $c$ be a sufficiently large constant that will be determined later.  Set $s = \lceil\frac{n}{(ck)^q}\rceil$.  For the sake of contradiction, suppose $\chi$ does not produce a monochromatic block-monotone path of depth $k$ and block-size $s$.  For each element $v \in [n]$, we label $v$ with $f(v) = (b_1,\ldots, b_q)$, where $b_i$ denotes the depth of the longest block-monotone path with block-size $s$ in color $i$, ending at $v$.  By our assumption, we have $0\leq b_i \leq k-1$, which implies that there are at most $k^q$ distinct labels.  By the pigeonhole principle, there is a subset $V\subset [n]$ of size at least $n/k^q$, s.t. the elements of $V$ all have the same label.  

By Lemma \ref{2path}, there are vertices $p_1,p_2 \in V$, a subset $V'\subset V$, and a color $\alpha$ s.t. $(p_1,V',p_2)$ is a monochromatic block-monotone path in color $\alpha$, with block-size $t \geq  \frac{|V|}{q3^{3q}}.$  By setting $c$ to be sufficiently large, we have
    \[t \geq \frac{|V|}{q3^{3q}} \geq \frac{n}{k^qq3^{3q}} \geq \left\lceil\frac{n}{(ck)^q}\right\rceil =  s.\]
However, this contradicts the fact that $f(p_1) = f(p_2)$, since the longest supported monotone path with block-size $s$ in color $\alpha$ ending at vertex $p_1$ can be extended to a longer one ending at $p_2$.  This completes the proof.  \end{proof}

\begin{proof}[Proof of Theorem \ref{main}]
Let $A=(a_1,\dots,a_n)$ be a sequence of distinct real numbers. Let $\chi$ be a red/blue coloring of the pairs of $A$ s.t. for $i<j$, we have $\chi(a_i,a_j) = $ red if $a_i < a_j$ and $\chi(a_i,a_j) = $ blue if $a_i > a_j$. In other words, we color the increasing pairs by red and the decreasing pairs by blue.

If $n< (ck)^2$, notice that $n/(ck)^2<1$. By our assumption $n>(k-1)^2$, the classical Erd{\H o}s-Szekeres theorem gives us a monotone subsequence in $A$ of length at least $k$, which can be regarded as a block-monotone subsequence of depth at least $k$ and block-size $s=1>n/(ck)^2$.

If $n\geq (ck)^2$, by Theorem~\ref{path}, there is a monochromatic block-monotone path of depth $k$ and block-size $s \geq n/(ck)^2$ in the complete graph on $A$, which can be regarded as a block-monotone subsequence of $A$ with the claimed depth and block-size.

Now we focus on computing such a block-monotone subsequence. If $n< (ck)^2$, it suffices to compute the longest monotone subsequence of $A$. It's well-known that the longest increasing subsequence can be computed within $O(n\log n)$ time, see \cite{fredman1975computing}, so we are done with this case.

If $n\geq (ck)^2$, we set $s=\lceil n/(ck)^2 \rceil$. We call a pair $(a_i,a_j)$ \emph{$s$-gapped} if there exist $s$ other entries $a_x$ with $i<x<j$ satisfying $a_i<a_x<a_j$ or $a_i>a_x>a_j$.
We describe an $O(n^2\log n)$-time algorithm that computes the longest increasing subsequence with consecutive entries $s$-gapped.

Firstly, we preprocess $A$ into a data structure s.t. we can answer within $O(\log n)$ time whether any given pair $(a_i,a_j)$ is $s$-gapped or not. The classical data structure for 2-D orthogonal range counting works for our purpose and its preprocessing time is $O(n\log n)$, see Exercise~5.10 in \cite{deberg2008computational}.

Next, let $l(i)$ be the length of the longest increasing subsequence of $A$ with consecutive entries $s$-gapped ending at $a_i$. We compute each $l(i)$ as $i$ proceeds from $1$ to $n$ as follows: After $l(1),\dots, l(i-1)$ are all determined, we have
    \[l(i)=\max_{j<i}\{l(j);(a_j,a_i)\text{ is $s$-gapped}\}+1.\]
Here, we consider $\max(\emptyset):=0$. Hence we can compute $l(i)$ by checking which pairs in $\{(a_j,a_i);j<i\}$ are $s$-gapped using our preprocessed data structure. Clearly, this computation of all $l(i)$ takes $O(n^2\log n)$ time.

While computing $l(i)$, let the algorithm record $p(i)$, which is the index $j<i$ with the largest $l(j)$ s.t. $(a_j,a_i)$ is $s$-gapped. This recording process won't increase the magnitude of time complexity. After all $l(i)$ and $p(i)$ are determined, we find the index $i_1$ with the largest $l(i)=:L$, and inductively set $i_{j+1}=p(i_j)$ for $j\in [L-1]$. Then $a_{i_L},a_{i_{L-1}},\ldots,a_{i_1}$ is the longest increasing subsequence of $A$ with consecutive entries $s$-gapped.

Let's return to computing the block-monotone subsequence. By the previous argument on block-monotone paths, there exists a monotone subsequence $S\subset A$ with consecutive entries $s$-gapped whose length is at least $k+1$. We can use the algorithm above to compute $S$ within $O(n^2\log n)$ time. Clearly, the entries of $A$ ``gapped'' by consecutive entries of $S$ form a block-monotone subsequence as claimed, and they can be found within $O(n)$ time. Hence we conclude the theorem.
\end{proof}

\begin{remark}\label{extreme_construction_remark}
For each $k,q,s>0$, the simple construction below shows Theorem~\ref{path} is tight up to the constant factor $c^q$. We first construct $K(k,q)$, for each $k$ and $q$, a $q$-colored complete graph on $[k^q]$, whose longest monochromatic monotone path has length $k$: $K(k,1)$ is just a monochromatic copy of the complete graph on $[k]$. To construct $K(k,q)$ from $K(k,q-1)$, take $k$ copies of $K(k,q-1)$ with the same set of $q-1$ colors, place them in order and color the remaining edges by a new color. Now replace each point in $K(k,q)$ by a cluster of $s$ points, where within each cluster one can arbitrarily color the edges. The resulting $q$-colored complete graph has no $k$ subsets $V_1,V_2,\dots,V_k\subset [n]$ each of size $s+1$ and edges between them monochromatic, otherwise $K(k,q)$ would have a monochromatic monotone path with length larger than $k$.

One example of the sharpness of the classical Erd{\H o}s-Szekeres theorem is the sequence
    \[S(k)=(k,k-1,\dots,1,2k,2k-1,\dots,k+1,\ldots,k^2,k^2-1,k(k-1)+1).\]
Notice that if we color the increasing pairs of $S(k)$ by red and the decreasing pairs of $S(k)$ by blue, we obtain the graph $K(k,2)$. If we replace each entry $s_i\in S(k)$ by a cluster of $s$ distinct real numbers very close to $s_i$, we obtain an example showing that Theorem~\ref{main} is asymptotically best possible.
\end{remark}

\section{Block-monotone sequence partition}\label{sec:partition}

This section is devoted to the proof of Theorem \ref{partition}. We shall consider this problem geometrically by identifying each entry $a_i$ of a given sequence $A=(a_i)_{i=1}^n$ as a planar point $(i,a_i)\in \mathbb{R}^2$. As we consider sequences of distinct real numbers, throughout this section, we assume that all point sets have the property that no two members share the same $x$-coordinate or the same $y$-coordinate.

Thus, we analogously define block-monotone point sets as follows: A set of $ks$ planar points is said to be \emph{block-increasing} (resp. \emph{block-decreasing}) with \emph{depth} $k$ and \emph{block-size} $s$ if it can be written as $\{(x_i,y_i)\}_{i=1}^{ks}$ s.t. $x_i<x_{i+1}$ for all $i$ and every sequence $(y_{i_1},y_{i_2},\dots, y_{i_k})$, for $(j-1)s<i_j\leq js$, is increasing (resp. decreasing). We say that a point set is \emph{block-monotone} if it's either block-increasing or block-decreasing. For each $j\in [k]$ we call the subset $\{(x_i,y_i)\}_{i=(j-1)s+1}^{js}$ the \emph{$j$-th block} of this block-monotone point set.

Hence, Theorem~\ref{partition} immediately follows from the following.
\begin{theorem}\label{partition_pointset}
For any positive integer $k$, every finite planar point set can be partitioned into at most $O(k^2\log k)$ block-monotone point subsets of depth at least $k$ and a remaining set of size at most $(k-1)^2$.
\end{theorem}

Given a point set $P\subset \mathbb{R}^2$, let \begin{align}
    U(P):=\{(x,y)\in \mathbb{R}^2; y> y',\ \forall (x',y')\in P\}, \tag{up}\\
    D(P):=\{(x,y)\in \mathbb{R}^2; y< y',\ \forall (x',y')\in P\}, \tag{down}\\
    L(P):=\{(x,y)\in \mathbb{R}^2; x< x',\ \forall (x',y')\in P\}, \tag{left}\\
    R(P):=\{(x,y)\in \mathbb{R}^2; x> x',\ \forall (x',y')\in P\}. \tag{right}
\end{align}

Our proof of Theorem~\ref{partition_pointset} relies on the following definitions. The constant $c$ below (and throughout this section) is from Theorem~\ref{path}. See Figure~\ref{fig_confpatt} for an illustration.
\begin{definition}\label{def_conf}
A point set $P$ is said to be a \emph{$(k,t)$-configuration} if $P$ can be written as a disjoint union of subsets $P = Y_1\cup Y_2\cup \cdots\cup Y_{2t+1}$ s.t.
\begin{itemize}
    \item $\forall i\in [t]$, $Y_{2i}$ is a block-monotone point set of depth $k$ and block-size at least $|Y_{2j+1}|/(3ck)^2$ for all $j\in \{0\}\cup[t]$;
    \item either $\bigcup_{j=i+1}^{2t+1} Y_j$ is located entirely in $R(Y_i)\cap U(Y_i)$ for all $i\in [2t]$, or $\bigcup_{j=i+1}^{2t+1} Y_j$ is located entirely in $R(Y_i)\cap D(Y_i)$ for all $i\in [2t]$.
\end{itemize}
\end{definition}

\begin{definition}\label{def_patt}
A point set $P$ is said to be a \emph{$(k,l,t)$-pattern} if $P$ can be written as a disjoint union of subsets $P = S_1\cup S_2\cup\dots\cup S_l\cup Y$ s.t.
\begin{itemize}
    \item  $Y$ is a $(k,t)$-configuration;
    \item $\forall i\in [l]$, $S_i$ is a block-monotone point set of depth $k$ and block-size at least $|Y|/(3ck)^2$;
    \item $\forall i\in [l]$, the set $(\bigcup_{j=i+1}^l S_j) \cup Y$ is located entirely in one of the following regions: $U(S_i)\cap L(S_i)$, $U(S_i)\cap R(S_i)$, $D(S_i)\cap L(S_i)$, and $D(S_i)\cap R(S_i)$.
\end{itemize}
\end{definition}

\begin{figure}[ht]
    \centering
    \includegraphics{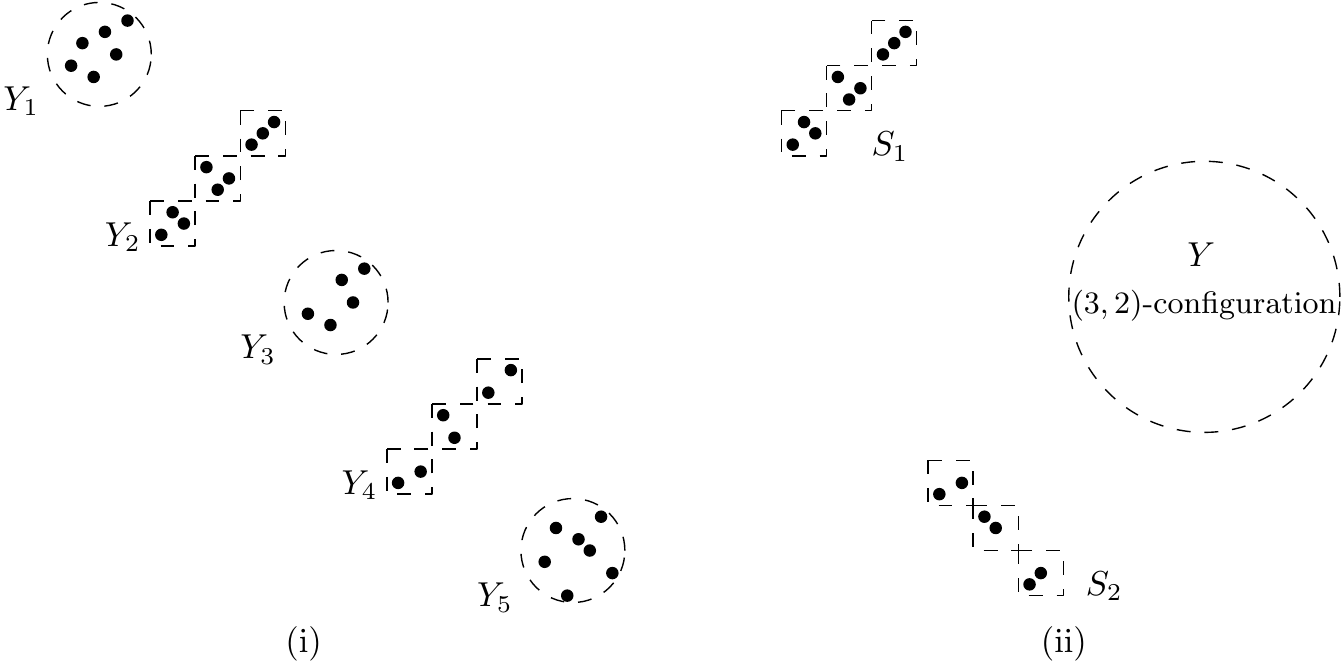}
    \caption{(i) a $(3,2)$-configuration. (ii) a $(3,2,2)$-pattern.}
    \label{fig_confpatt}
\end{figure}

If a planar point set $P$ is a $(k,4k,t)$-pattern or a $(k,l,k)$-pattern, the next two lemmas state that we can efficiently partition $P$ into few block-monotone point sets and a small remaining set.

\begin{lemma}\label{partition_lemma_1}
If $P$ is a $(k,4k,t)$-pattern, then $P$ can be partitioned into $O(k\log k)$ block-monotone point sets of depth at least $k$ and a remaining set of size $O(k^2)$.
\end{lemma}
\begin{lemma}\label{partition_lemma_2}
If $P$ is a $(k,l,k)$-pattern, then $P$ can be partitioned into $O(k^2\log k+l)$ block-monotone point sets of depth at least $k$ and a remaining set of size $O(k^3)$.
\end{lemma}

Starting with an arbitrary point set $P$, which can be regarded as a $(k,0,0)$-pattern, we will repeatedly apply the following lemma until $P$ is partitioned into few block-monotone point sets, a set $P'$ that is either a $(k,4k,t)$-pattern or a $(k,l,k)$-pattern, and a small remaining set.
\begin{lemma}\label{partition_lemma_3}
For $l < 4k$ and $t < k$, a  $(k,l,t)$-pattern $P$ can be partitioned into $O(k\log k)$ block-monotone point sets of depth at least $k$, a point set $P'$, and a remaining set of size $O(k^2)$,
s.t. either (i) $|P'|\leq  k(3k-1)^2$; or (ii) $P'$ is a $(k,l,t+1)$-pattern; or (iii) $P'$ is a $(k,l+t,0)$-pattern. Moreover, $P'$ always satisfies either (i) or (ii) when $t=0$.
\end{lemma}

Before we prove the lemmas above, let us use them to prove Theorem \ref{partition_pointset}.
\begin{proof}[Proof of Theorem~\ref{partition_pointset}]
Let $P$ be the given point set. For $i\geq 0$, we inductively construct a partition $\mathcal{F}_i\cup \{P_i,E_i\}$ of $P$ s.t.
\begin{itemize}
    \item $P_i$ is a $(k,l_i,t_i)$-pattern or $|P_i| \leq k(3k-1)^2$, 
    \item $|E_i|=O(ik^2)$,
    \item $\mathcal{F}_i$ is a disjoint family of block-monotone point sets of depth at least $k$, and $|\mathcal{F}_i| = O(ik\log k)$.
\end{itemize}

We start with $P_0  = P$, which is a $(k,0,0)$-pattern, and $\mathcal{F}_0= E_0 = \emptyset$. Suppose we have constructed the $i$-th partition $\mathcal{F}_i\cup \{P_i,E_i\}$ of $P$. If $|P_i| \leq k(3k-1)^2$, or $l_i \geq 4k$, or $t_i \geq k$, we end this inductive construction process, otherwise, we construct the next partition $\mathcal{F}_{i+1}\cup \{P_{i+1},E_{i+1}\}$ as follows: According to Lemma~\ref{partition_lemma_3}, $P_i$ can be partitioned into $r_i=O(k\log k)$ block-monotone point sets with depth at least $k$, denoted as $\{A_{i,1},\dots, A_{i,r_i}\}$, a point set $P_i'$, and a remaining set $E_i'$ of size $O(k^2)$. We define $\mathcal{F}_{i + 1}:=\mathcal{F}_i\cup\{A_{i,1},\dots, A_{i,r_i}\}$, $P_{i + 1}:=P_i'$, and $E_{i+1}:=E_i\cup E_i'$. Clearly, we have $|\mathcal{F}_{i + 1}| = |\mathcal{F}_{i}|  + r_i = O((i + 1)k\log k)$ and $|E_{i+1}|=|E_i|+|E_i'|=O((i+1)k^2)$ as claimed. By Lemma~\ref{partition_lemma_3}, we have either (i) $|P_{i+1}|\leq  k(3k-1)^2$, so the construction ends; or (ii) $P_{i+1}$ is a $(k,l_i,t_i+1)$-pattern, so $l_{i+1}=l_i$ and $t_{i+1}=t_{i}+1$; or (iii) $P_{i+1}$ is a $(k,l_i+t_i,0)$-pattern, so $l_{i+1}=l_i+t_i$ and $t_{i+1}=0$. Moreover, when $t_i=0$, Lemma~\ref{partition_lemma_3} guarantees that $P_{i+1}$ always satisfies (i) or (ii).

Let $\mathcal{F}_w\cup \{P_w,E_w\}$ be the last partition of $P$ constructed in this process. Here, $P_w$ is a $(k,l_w,t_w)$-pattern. We must have either $|P_w| \leq k(3k-1)^2$, or $l_w \geq 4k$, or $t_w \geq k$. Since $t_{i+1}\leq t_i+1$ and $l_{i+1}\leq l_{i}+t_i$ for all $i$, we have $t_w\leq k$ and $l_w\leq 5k$. Since we always construct the $(i+1)$-th partition with $P_{i+1}$ satisfying either (i) or (ii) when $t_i = 0$, the sum $l_i + t_i$ always increases by at least $1$ after two inductive steps. So we have $w/2\leq t_w+l_w\leq 6k$ and hence $w\leq 12k$.

Now we handle $\mathcal{F}_w\cup \{P_w,E_w\}$ based on how the construction process ends.

If the construction process ended with $|P_w| \leq k(3k-1)^2$, we define $E_{w+1}=E_w\cup P_w$ and $\mathcal{F}_{w + 1}=\mathcal{F}_{w}$. Since $w\leq 12k$, we have $|\mathcal{F}_{w + 1}|=O(k^2\log k)$ and $|E_{w+1}|=O(k^3)$.

If the construction process ended with $l_w\geq 4k$, by Definition~\ref{def_patt}, we can partition $P_w$ into $l_w-4k$ many block-monotone point sets of depth $k$, denoted as $\{S_1,\dots, S_{l_w-4k}\}$, and a $(k,4k,t_w)$-pattern $P'_w$. Then, by Lemma~\ref{partition_lemma_1}, $P'_w$ can be partitioned into $r_w=O(k\log k)$ block-monotone point sets of depth at least $k$, denoted as $\{A_{w,1},\dots, A_{w,r_w}\}$, and a remaining set $E'_w$ of size $O(k^2)$. We define $E_{w+1}=E_w\cup E'_w$ and
    \[\mathcal{F}_{w + 1}=\mathcal{F}_{w}\cup\{S_1,\dots, S_{l_w-4k},A_{w,1},\dots, A_{w,r_w}\}.\]
Using $w\leq 12k$, $l_w\leq 5k$, and other bounds we mentioned above, we can check that $|\mathcal{F}_{w + 1}|=O(k^2\log k)$ and $|E_{w+1}|=O(k^3)$.

If the construction process ended with $t_w\geq k$, we actually have $t_w=k$ and $l_w<4k$. By Lemma~\ref{partition_lemma_2}, we can partition $P_w$ into $r_w=O(k^2\log k+l_{w})$ block-monotone point sets of depth at least $k$, denoted as $\{A_{w,1},\dots, A_{w,r_w}\}$, and a remaining set $E'_w$ of size $O(k^3)$. We define $E_{w+1}=E_w\cup E'_w$ and $ \mathcal{F}_{w + 1}=\mathcal{F}_{w}\cup\{A_{w,1},\dots, A_{w,r_w}\}$. Again, we have $|\mathcal{F}_{w + 1}|=O(k^2\log k)$ and $|E_{w+1}|=O(k^3)$.

Overall, we can always obtain a partition $\mathcal{F}_{w + 1}\cup \{E_{w + 1}\}$ of $P$ with $|\mathcal{F}_{w + 1}|=O(k^2\log k)$ and $|E_{w+1}|=O(k^3)$. Using the classical Erd{\H o}s-Szekeres theorem, we can always find a monotone sequence of length at least $k$ in $E_{w+1}$ when $|E_{w+1}|>(k-1)^2$. By a repeated application of this fact, we can partition $E_{w+1}$ into $O(k^2)$ block-monotone point sets of depth $k$ and block-size~$1$, and a remaining set $E$ of size at most $(k-1)^2$. We define $\mathcal{F}$ to be the union of $\mathcal{F}_{w+1}$ and these block-monotone sequences. The partition $\mathcal{F}\cup \{E\}$ of $P$ has the desired properties, completing our proof.
\end{proof}

We now give proofs for Lemmas \ref{partition_lemma_1}, \ref{partition_lemma_2}, and \ref{partition_lemma_3}. We need the following two facts.  

\begin{fact}\label{partition_fact_1}
For any positive integer $k$, every point set $P$ can be partitioned into $O(k\log k)$ block-monotone point sets of depth $k$ and a remaining set $P'$ with $|P'|\leq \max\{|P|/k,(k-1)^2\}$.
\end{fact}
This fact can be established by repeatedly using Theorem~\ref{main} to pull out block-monotone point sets and applying the elementary inequality $(1-x^{-1})^{x\log x}\leq x^{-1}$ for any $x>1$.

\begin{fact}\label{partition_fact_2}
For any positive integer $k$ and $m$, every block-monotone point set $P$ with depth $k$ and $|P|\geq m$ can be partitioned into a block-monotone point set of depth $k$, a subset of size exactly $m$, and a remaining set of size less than $k$.
\end{fact}
This fact can be established by taking out $\lceil m/k\rceil$ points from each block of $P$. Then we have taken out $k\cdot\lceil m/k\rceil=m+r$ points, where $0\leq r<k$.

\begin{proof}[Proof of Lemma~\ref{partition_lemma_1}]
Write the given $(k,4k,t)$-pattern $P=S_1\cup \dots \cup S_{4k}\cup Y$ as in Definition~\ref{def_patt}. By definition, each block-monotone point set $S_i$ is contained in one of the $4$ regions: $U(Y)\cap L(Y)$, $U(Y)\cap R(Y)$, $D(Y)\cap L(Y)$, and $D(Y)\cap R(Y)$. By the pigeonhole principle, there are $k$ point sets among $S_1,\dots,S_{4k}$ all contained in one of the regions above. Without loss of generality, we assume there are $k$ among them all contained in $U(Y)\cap L(Y)$. In fact, we can further assume that these point sets are $S_1,\dots,S_k$ as the proof also works for other cases.

We have $S_{i'}\subset D(S_{i})\cap R(S_{i})$ for all $1\leq i<i'\leq k$. Indeed, since $Y\subset D(S_{i})\cap R(S_{i})$, Definition~\ref{def_patt} guarantees that $(\bigcup_{j=i+1}^k S_j) \cup Y$ is contained in $D(S_{i})\cap R(S_{i})$ and, in particular, $S_{i'}$ is contained in this region. See Figure~\ref{fig_partition_lemma} for an illustration.

\begin{figure}
    \centering
    \includegraphics{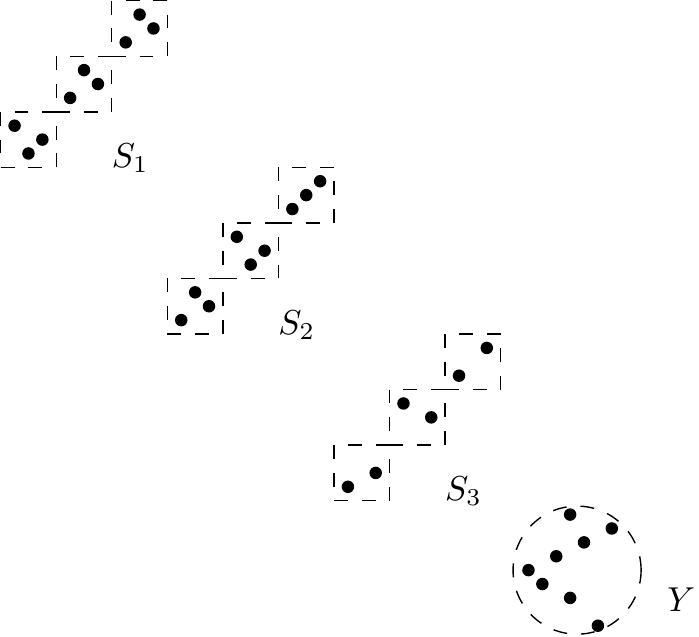}
    \caption{In proof of Lemma~\ref{partition_lemma_1}, $S_{i'}\subset D(S_{i})\cap R(S_{i})$ for $i<i'$.}
    \label{fig_partition_lemma}
\end{figure}

Now by Fact~\ref{partition_fact_1}, we can partition $Y$ into $\{A_1,\dots, A_r, Y'\}$, where $r=O(k\log k)$, s.t. each $A_i$ is block-monotone of depth $9c^2k$, and either $|Y'|\leq |Y|/(9c^2k)$ or $|Y'|\leq (9c^2k-1)^2$. If $|Y'|\leq (9c^2k-1)^2$, we have partitioned $P$ into $O(k\log k)$ block-monotone point sets of depth at least $k$, which are $\{A_1,\dots, A_r, S_1,\dots, S_{4k}\}$, and a remaining set $Y'$ of size $O(k^2)$, as wanted.

If $|Y'|\leq |Y|/(9c^2k)$, by Definition~\ref{def_patt} we have $|Y'| \leq |S_i|$ for $i\in [k]$. We can apply Fact~\ref{partition_fact_2} with $m:=|Y'|$ to obtain a partition $S_i=S'_i\cup B_i\cup E_i$ where $S'_i$ is block-monotone of depth $k$, $|B_i|=|Y'|$, and $|E_i|\leq k$. Observe that $X:=B_1\cup B_2\cup \dots \cup B_k \cup Y'$ is block-monotone of depth $k+1$ by its construction. Then we have partitioned $P$ into $O(k\log k)$ many block-monotone point sets, which are $\{A_1,\dots, A_r, S'_1,\dots, S'_k, S_{k+1},\dots, S_{4k}, X\}$, and a remaining set $E:=\bigcup_{i=1}^k E_i$ of size $O(k^2)$, as wanted.
\end{proof}
\begin{proof}[Proof of Lemma~\ref{partition_lemma_2}]
Write the given $(k,l,k)$-pattern $P=S_1\cup \dots\cup S_l \cup Y$ as in Definition~\ref{def_patt} and the $(k,k)$-configuration $Y=Y_1\cup \dots \cup Y_{2k+1}$ as in Definition~\ref{def_conf}. Since each $S_i$ is block-monotone of depth $k$, it suffices to partition $Y$ into $O(k^2\log k)$ many block-monotone point sets of depth at least $k$ and a remaining set of size $O(k^3)$.

For each $j\in \{0\}\cup [k]$, we apply Fact~\ref{partition_fact_1} to obtain a partition of $Y_{2j+1}$ into $O(k\log k)$ many block-monotone point sets of depth $9c^2k$ and a remaining set $Y'_{2j+1}$ of size either at most $|Y_{2j+1}|/(9c^2k)$ or at most $(9c^2k-1)^2$. We can apply Fact~\ref{partition_fact_1} again to partition $Y'_{2j+1}$ into $O(k\log k)$ many block-monotone point sets of depth $k+1$ and a remaining set $Y''_{2j+1}$ with\begin{align}\label{eq_partition_lemma_2}
    |Y''_{2j+1}|\leq \max\{|Y_{2j+1}|/(9c^2k(k+1)),(9c^2k-1)^2\}.
\end{align} Denote the block-monotone point sets produced in this process as $\{A_{j,1},\dots, A_{j,r_j}\}$ where $r_j=O(k\log k)$.

Next we denote $J_1:=\{j\in \{0\}\cup [k]; |Y_{2j+1}''|>(9c^2k-1)^2\}$ and $J_2:=(\{0\}\cup [k])\setminus J_1$. For each $j\in J_1$ and $i\in [k]$, we must have
    \[|Y_{2j+1}''|\leq |Y_{2j+1}|/(9c^2k(k+1))\leq |Y_{2i}|/(k+1),\]
where the second inequality is by Definition~\ref{def_conf}. Hence $|Y_{2i}|\geq |\bigcup_{j\in J_1} Y''_{2j+1}|$. We can apply Fact~\ref{partition_fact_2} with $m:=|\bigcup_{j\in J_1} Y''_{2j+1}|$ to obtain a partition $Y_{2i}=Y'_{2i}\cup B_{i}\cup E_{i}$ where $Y'_{2i}$ is block-monotone of depth $k$, $|B_{i}|=m$, and $|E_{i}|\leq k$. Since $|B_{i}|=|\bigcup_{j\in J_1} Y''_{2j+1}|$, we can take a further partition $B_{i}=\bigcup_{j\in J_1} B_{j,i}$ with $|B_{j,i}|=|Y''_{2j+1}|$ for each $j\in J_1$. Then we observe that \[X_{j}:=B_{j,1}\cup \dots \cup B_{j,j}\cup Y''_{2j+1}\cup B_{j,j+1}\cup \dots \cup B_{j,k}\] is block-monotone of depth $k+1$ for each $j\in J_1$ by its construction.

Finally, let $E:=(\bigcup_{i=1}^k E_i)\cup(\bigcup_{j\in J_2}Y''_{2j+1})$, it easy to check that $|E|=O(k^3)$. So we have partitioned $Y$ into $O(k^2\log k)$ many block-monotone point sets, which are $(\bigcup_{j=0}^k \{A_{j,1},\dots, A_{j,r_j}\})\cup \{X_j\}_{j\in J_1} \cup \{Y'_{2},Y'_4,\dots, Y'_{2k}\}$, and a remaining set $E$ of size $O(k^3)$, as wanted.
\end{proof}
\begin{proof}[Proof of Lemma~\ref{partition_lemma_3}]
Write the given $(k,l,t)$-pattern $P=S_1\cup \dots\cup S_l \cup Y$ as in Definition~\ref{def_patt} and its constituent $(k,t)$-configuration $Y=Y_1\cup \dots \cup Y_{2t+1}$ as in Definition~\ref{def_conf}. Denote a set with largest size among $\{Y_{2j+1}; j\in \{0\}\cup [t]\}$ as $Y_{i_0}$. By Definition~\ref{def_conf}, we can assume without loss of generality that\begin{equation}\label{eq:partition_lemma_3}
    \bigcup\nolimits_{j=i+1}^{2t+1} Y_j \subset R(Y_i)\cap U(Y_i) \text{ for all } i\in [2t]
\end{equation}

If $|Y_{i_0}|\leq (3k-1)^2$, we can partition $P$ into $l+t=O(k)$ many block-monotone point sets of depth $k$, which are $\{S_1,\dots, S_l,Y_2,Y_4,\dots, Y_{2t}\}$, and a set $P':=\bigcup_{j=0}^t Y_{2j+1}$ of size at most $k(3k-1)^2$ since $t<k$. So we conclude the lemma with $P'$ as described in (i).

Now we assume $|Y_{i_0}|>(3k-1)^2$. Apply Theorem~\ref{main} to extract a block-monotone point set $X\subset Y_{i_0}$ of depth $3k$ and block-size at least $|Y_{i_0}|/(3ck)^2$ and name the $i$-th block of $X$ as $B_i$ for $i\in [3k]$. Notice that $X$ splits into three parts\begin{equation*}
    X_1:=B_1\cup \dots\cup B_k,\ X_2:=B_{k+1}\cup \dots\cup B_{2k},\ X_3:=B_{2k+1}\cup\dots\cup B_{3k}. 
\end{equation*} Our proof splits into two cases: $X$ being block-increasing or $X$ being block-decreasing.

\medskip

\noindent \emph{Case 1.} Suppose $X$ is block-increasing, we define\begin{equation*}
    P'=(S_1\cup \dots\cup S_l)\cup (Y_2\cup Y_4\cup\dots\cup Y_{i_0-1})\cup (Y_{2t}\cup Y_{2t-2}\cup\dots\cup Y_{i_0+1})\cup (Y_{i_0}\setminus X).
\end{equation*} By Definition~\ref{def_patt}, we can check that $P'$ is a $(k,k+l,0)$-pattern with $Y_{i_0}\setminus X$ being its constituent $(k,0)$-configuration. Let $Z_1:=Y_1\cup Y_3\cup\dots\cup Y_{i_0-2}$ and $Z_3:=Y_{i_0+2}\cup Y_{i_0+4}\cup \dots \cup Y_{2t+1}$. We claim that $X_i\cup Z_i$ can be partitioned into $O(k\log k)$ block-monotone point sets of depth at least $k$ and a remaining set of size $O(k^2)$ for $i=1,3$. Given this claim and the fact that $P=P'\cup X\cup Z_1\cup Z_3$, we conclude the lemma with $P'$ as described in (iii).

Now we justify this claim for $X_1\cup Z_1$ and the justification for $X_3\cup Z_3$ is similar. By an argument similar to \eqref{eq_partition_lemma_2}, we can apply Fact~\ref{partition_fact_1} three times to partition $Z_1$ into $\{A_{1},\dots,A_{r},Z_1'\}$, where $r=O(k\log k)$, s.t. each $A_{i}$ is block-monotone of depth at least $k$ and $|Z_1'|\leq\max\{|Z_1|/(9c^2k^3), (9c^2k-1)^2\}$. If $|Z_1'|\leq (9c^2k-1)^2$, we have partitioned $X_1\cup Z_1$ into $O(k\log k)$ block-monotone point sets of depth at least $k$, which are $\{A_{1},\dots,A_{r},X_1\}$, and a remaining set $Z_1'$ of size $O(k^2)$ as claimed.

If $|Z_1'|\leq |Z_1|/(9c^2k^3)$, noticing that $|Z_1|\leq t|Y_{i_0}|\leq k|Y_{i_0}|$, we have $|Z_1'|\leq |Y_{i_0}|/(3ck)^2\leq |B_i|$ for all $i\in [k]$. We can take a partition $B_i=B_{i}'\cup B_i''$ with $|B_i'|=|Z_1'|$. We observe that $X_1':=Z_1'\cup B_1'\cup \dots\cup B_{k}'$ is block-increasing of depth $k+1$ and $X_1'':=B''_{1}\cup \dots\cup B''_{k}$ is block-increasing of depth $k$ by their constructions. So we have partitioned $X_1\cup Z_1$ into $O(k\log k)$ block-monotone point sets of depth at least $k$, which are $\{A_1,\dots, A_r, X_1', X_1''\}$, as claimed.

\medskip

\noindent \emph{Case 2.} Suppose $X$ is block-decreasing, we choose two points in the following regions:\begin{align*}
    &(x_1, y_1)\in R(B_k)\cap D(B_k)\cap L(B_{k+1})\cap U(B_{k+1}),\\
    &(x_2, y_2)\in R(B_{2k})\cap D(B_{2k})\cap L(B_{2k+1})\cap U(B_{2k+1}).
\end{align*} Also we require $x_1$ or $x_2$ isn't the $x$-coordinate of any element in $P$, and $y_1$ or $y_2$ isn't the $y$-coordinate of any element in $P$. We use the lines $x=x_i$ and $y=y_i$ for $i=1,2$ to divide the plane into a $3\times 3$ grid and label the regions $R_i,i=1,\dots,9$ as in Figure~\ref{fig:partition_9region}.

\begin{figure}[ht]
    \centering
    \includegraphics{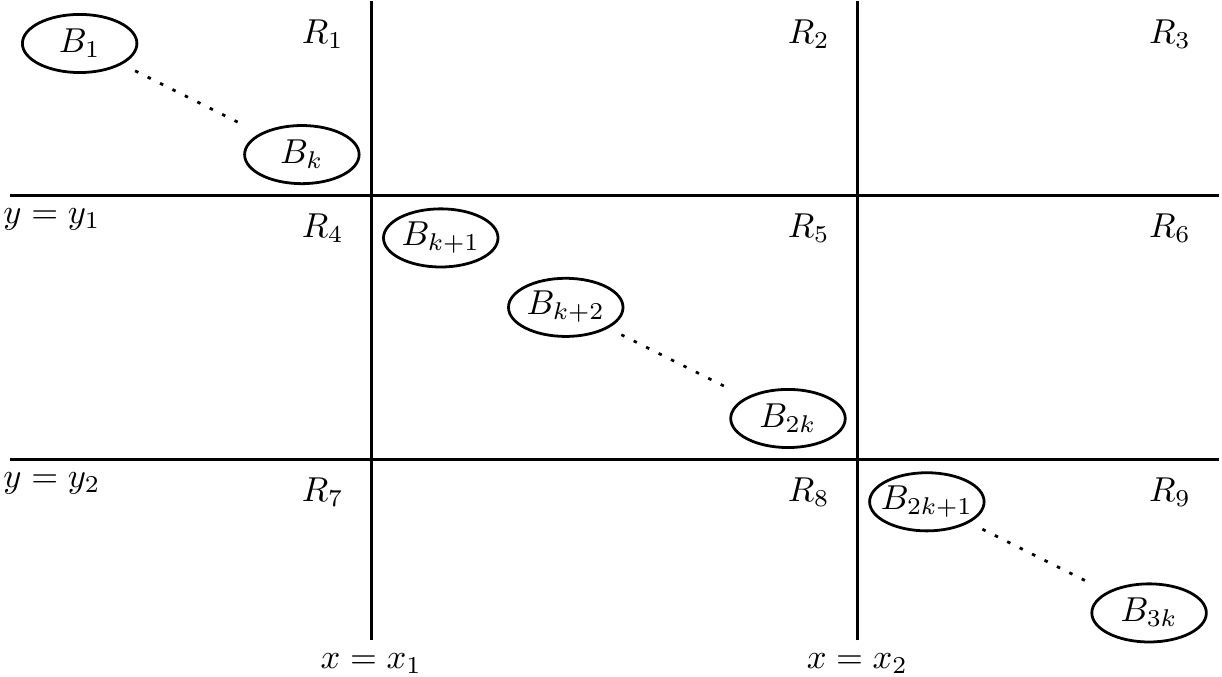}
    \caption{Division of the plane into $9$ regions according to $(x_i,y_i),i=1,2$. Each ellipse represents a cluster of points as defined in the proof.}
    \label{fig:partition_9region}
\end{figure}

We define\begin{equation*}
    Y':=Y_1\cup Y_2\cup \dots\cup Y_{i_0-1}\cup (R_7\cap Y_{i_0})\cup X_2\cup (R_3\cap Y_{i_0})\cup Y_{i_0+1}\cup \dots \cup Y_{2t+1}.
\end{equation*} Using condition \eqref{eq:partition_lemma_3} and Definition~\ref{def_conf}, we can check that $Y'$ is a $(k,t+1)$-configuration. And $P':=S_1\cup \dots\cup S_l\cup Y'$ is a $(k,l,t+1)$-pattern according to Definition~\ref{def_patt}. Let $Z_1:=(Y_{i_0}\setminus X)\cap (R_5\cup R_6\cup R_8\cup R_9)$ and $Z_2:=(Y_{i_0}\setminus X)\cap (R_1\cup R_2\cup R_4)$. We claim that $X_i\cup Z_i$ can be partitioned into $O(k\log k)$ block-monotone point sets of depth at least $k$ and a remaining set of size $O(k^2)$ for $i=1,3$. Given this claim and the fact that $P=P'\cup X_1\cup X_3\cup Z_1\cup Z_3$, we conclude the lemma with $P'$ as described in (ii).

Now we justify this claim for $X_1\cup Z_1$ and the justification for $X_3\cup Z_3$ is similar. By an argument similar to \eqref{eq_partition_lemma_2}, we can apply Fact~\ref{partition_fact_1} twice to partition $Z_1$ into $\{A_{1},\dots,A_{r},Z_1'\}$, where $r=O(k\log k)$, s.t. each $A_{i}$ is block-monotone of depth at least $k$ and $|Z_1'|\leq\max\{|Z_1|/(3ck)^2, (9c^2k-1)^2\}$. If $|Z_1'|\leq (9c^2k-1)^2$, we have partitioned $X_1\cup Z_1$ into $O(k\log k)$ block-monotone point sets of depth at least $k$, which are $\{A_{1},\dots,A_{r},X_1\}$, and a remaining set $Z_1'$ of size $O(k^2)$ as claimed.

If $|Z_1'|\leq |Z_1|/(3ck)^2$, since $|Z_1|\leq |Y_{i_0}|$, we have $|Z_1'|\leq |Y_{i_0}|/(3ck)^2\leq |B_i|$ for all $i\in [k]$. We can take a partition $B_i=B_{i}'\cup B_i''$ with $|B_i'|=|Z_1'|$. We observe that $X_1':=B_1'\cup \dots\cup B_{k}'\cup Z'_1$ is block-decreasing of depth $k+1$ and $X_1'':=B''_{1}\cup \dots\cup B''_{k}$ is block-decreasing of depth $k$ by their constructions. So we have partitioned $X_1\cup Z_1$ into $O(k\log k)$ block-monotone point sets of depth at least $k$, which are $\{A_1,\dots, A_r, X_1', X_1''\}$, as claimed.

\medskip

Finally, when $t=0$ is given in the hypothesis, the condition \eqref{eq:partition_lemma_3} and its opposite, i.e. $\bigcup_{j=i+1}^{2t+1} Y_j \subset R(Y_i)\cap D(Y_i)$ for all $i\in [2t]$, are both trivially true. Hence, when $|Y_{i_0}|>(3k-1)^2$, no matter whether $X$ is block-increasing or block-decreasing, we can always use the arguments in Case 2 to conclude the lemma as described in (ii).
\end{proof}

\section{Applications}\label{sec:applications}
\subsection{Mutually avoiding sets}
We devote this subsection to the proof of Theorem~\ref{avoiding}.  The proof is essentially the same as in \cite{aronov1991crossing}, but we include it here for completeness.  Given a non-vertical line $L$ in the plane, we denote $L^+$ to be the closed upper-half plane defined by $L$, and $L^-$ to be the closed lower-half plane defined by $L$. We need the following result, which is Lemma~1 in \cite{aronov1991crossing}.
\begin{lemma}\label{avoiding_lemma}
Let $P,Q\subset \mathbb{R}^2$ be two $n$-element point sets with $P$ and $Q$ separated by a non-vertical line $L$ and $P\cup Q$ in general position. Then for any positive integer $m\leq n$, there is another non-vertical line $H$ s.t. $|H^+\cap P|=|H^+\cap Q|=m$ or $|H^-\cap P|=|H^-\cap Q|=m$.
\end{lemma}
\begin{proof}[Proof of Theorem \ref{avoiding}]  Let $k$ be as given and $n > 24k^2$.  Let $P$ be an $n$-element point set in the plane in general position.  We start by taking a non-vertical line $L$ to partition the plane s.t. each half-plane contains $\lfloor \frac{n}{2} \rfloor$ points from $P$. Then by Lemma~\ref{avoiding_lemma}, we obtain a non-vertical line $H$ with, say, $H^+\cap(L^{+}\cap P)=H^+\cap (L^-\cap P)=\lfloor \frac{n}{6}\rfloor$. Next, we find a third line $N$, by first setting $N = H$, and then sweeping $N$ towards the direction of $H^-$, keeping it parallel with $H$, until $H^-\cap N^+\cap L^+$ or $H^-\cap N^+\cap L^-$ contains $\lfloor\frac{n}{6}\rfloor$ points from $P$. Without loss of generality, let us assume $Q:=P\cap (H^-\cap N^+\cap L^+)$ first reaches $\lfloor\frac{n}{6}\rfloor$ points, and the region $H^-\cap N^+\cap L^-$ has less than $\lfloor\frac{n}{6}\rfloor$ points from $P$. Hence, both $Q_l:=P\cap (H^+\cap L^-)$ and $Q_r:=P\cap (N^-\cap L^-)$ have at least $\lfloor \frac{n}{6}\rfloor$ points. See Figure~\ref{fig:avoiding1} for an illustration.

\begin{figure}[ht]
\centering
\includegraphics{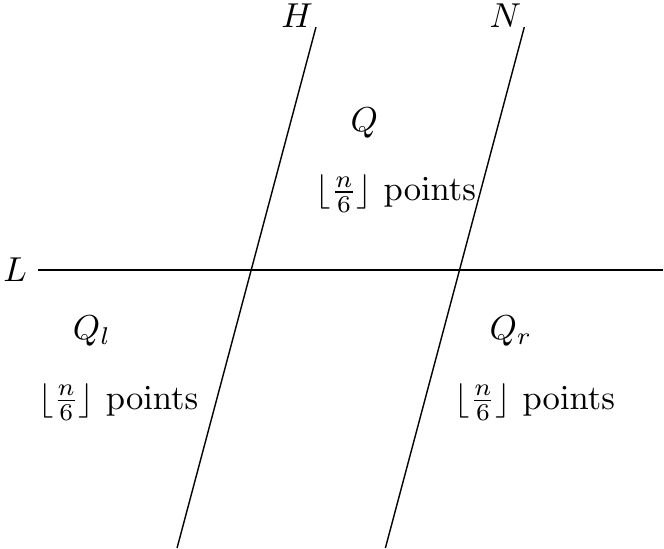}
\caption{The division of plane into regions according to $L,H,N$.}
\label{fig:avoiding1}
\end{figure}

We can apply an affine transformation so that $L$ and $H$ are perpendicular, and $N$ is on the right side of $H$. Think of $L$ as the $x$-axis, $H$ as the $y$-axis, and $N$ as a vertical line with a positive $x$-coordinate. After ordering the elements in $Q$ according to their $x$-coordinates, we apply Theorem \ref{main} to $Q$ to obtain disjoint subsets $Q_1,\dots,Q_{2k+1} \subset Q$ s.t. $(Q_1,\dots, Q_{2k+1})$ is block-monotone of depth $2k + 1$ and block-size $\Omega(n/k^2)$, where each entry represents its $y$-coordinate. Without loss of generality, we can assume it is block-decreasing, otherwise we can work with $Q_r$ rather than $Q_l$ in the following arguments.
\begin{figure}[ht]
   \centering
   \includegraphics{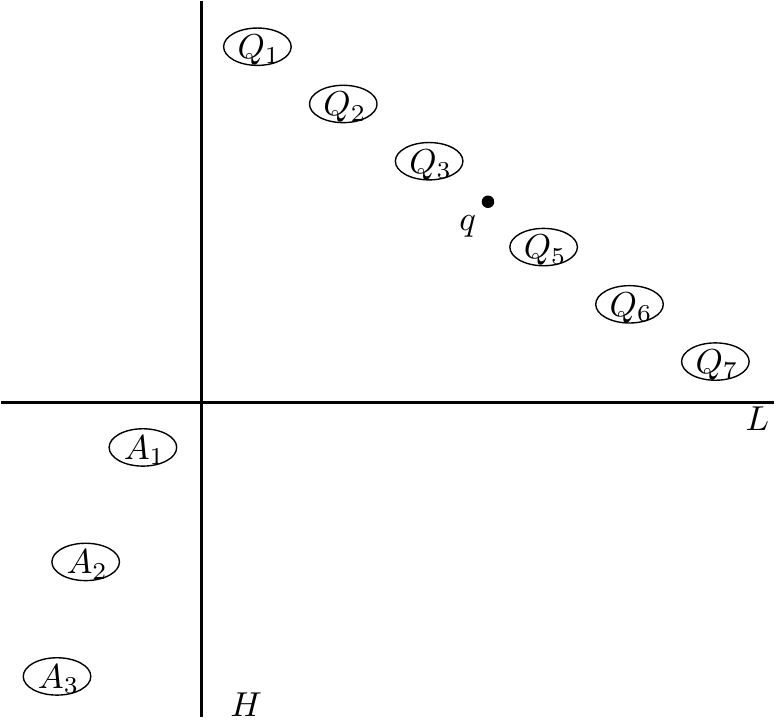}
    \caption{An example when $A_i$'s are increasing. Each ellipse represents a cluster of points as defined in the proof.}
    \label{fig:avoiding2}
\end{figure}

Now fix a point $q \in Q_{k+1}$. We express the points in $Q_l$ in polar coordinates $(\rho,\theta)$ with $q$ being the origin. We can assume no two points in $Q_l$ are at the same distance to $q$, otherwise a slight perturbation may be applied. By ordering the points in $Q_l$ with respect to $\theta$, in counter-clockwise order, we apply Theorem~\ref{main} to $Q_l$ to obtain disjoint subsets $A_1,\dots,A_k \subset Q_l$ s.t. $(A_1,\dots,A_k)$ is block-monotone of depth $k$ and block-size $\Omega(n/k^2)$, where each entry represents its distance to $q$. If it's block-decreasing, take $B_i=Q_{i}$ for $i\in[k]$, and if it's block-increasing, take $B_i=Q_{k +1+i}$. It is easy to check that the sets $\{A_1,\ldots, A_k\}$ and $\{B_1,\ldots, B_k\}$ have the claimed properties. See Figure~\ref{fig:avoiding2} for an illustration.\end{proof}

\subsection{Monotone biarc diagrams}
We devote this subsection to the proof of Theorem~\ref{biarc}. Our proof is constructive, hence implying a recursive algorithm for the claimed outcome.

We start by making the simple observation that our main results hold for sequences of (not necessarily distinct) real numbers, if the term ``\emph{block-monotone}'' now refers to being block-nondecreasing or block-nonincreasing. More precisely, a sequence $(a_1, a_2,\dots,a_{ks})$ of real numbers is said to be \emph{block-nondecreasing} (resp. \emph{block-nonincreasing}) with \emph{depth} $k$ and \emph{block-size} $s$ if every subsequence $(a_{i_1},a_{i_2},\dots, a_{i_k})$, for $(j-1)s<i_j\leq js$, is nondecreasing (resp. nonincreasing).

\begin{theorem}\label{partition_notdistinct}
For any positive integer $k$, every finite sequence of (not necessarily distinct) real numbers can be partitioned into at most $O(k^2\log k)$ block-monotone subsequences of depth at least $k$ upon deleting at most $(k-1)^2$ entries.
\end{theorem}

To see our main results imply the above variation, it suffices to slightly perturb the possibly equal entries of a given sequence until all entries are distinct. Algorithms for our main results can also be applied after such a perturbation.

We need the following lemma in \cite{yehuda1998partitioning} for Theorem~\ref{biarc}.
\begin{lemma}\label{biarc_lemma}
For any graph $G=(V,E)$ with $V = [n]$, there exists $b\in [n]$ s.t. both the induced subgraphs of $G$ on $\{1,2,\dots,b\}$ and $\{b+1,b+2,\dots,n\}$ have no more than $|E|/2$ edges.
\end{lemma}
\begin{proof}
For $U \subset [n]$, let $G_U$ denote the induced subgraph of $G$ on $U$.  Let $b$ be the largest among $[n]$ s.t. $E(G_{[b]})\leq \frac{|E|}{2}$, so $E(G_{[b+1]})>\frac{|E|}{2}$. Notice that $E(G_{[b+1]})$ and $E(G_{[n]\setminus[b]})$ are two disjoint subsets of $E$, so $E(G_{[n]\setminus[b]})\leq |E|- E(G_{[b+1]})<\frac{|E|}{2}$, as wanted.
\end{proof}
\begin{proof}[Proof of Theorem \ref{biarc}] We prove by induction on $|E|$. The base case when $|E| = 1$ is trivial.  For the inductive step, by the given order on $V$, we can identify $V$ with $[n]$. We find such a $b$ according to Lemma~\ref{biarc_lemma}. Consider the set $E'$ of edges between $[b]$ and $[n]\setminus [b]$.  By writing each edge $e \in E'$ as $(x,y)$, where $x\in [b]$ and $y\in [n]\setminus[b]$, we order the elements in $E'$ lexicographically: for $(x,y),(x',y') \in E$, we have $(x,y) < (x',y')$ when $x < x'$ or when $x = x'$ and $y < y'$.

Given the order on $E'$ described above, consider the sequence of right-endpoints in $E'$. We apply Theorem~\ref{partition_notdistinct} with parameter $k = \lceil\epsilon^{-1}\rceil$ to this sequence, to decompose it into $C_k=O(k^2\log k)$ many block-monotone sequences of depth $k$, upon deleting at most $(k-1)^2$ entries. For each block-monotone subsequence of depth $k$, we draw the corresponding edges on a single page as follows.  If the subsequence is block-nonincreasing of depth $k$ and block-size $s$, we draw the corresponding edges as semicircles above the spine.  Then, two edges cross only if they come from the same block.  Since there are a total of $\binom{ks}{2}$ pairs of edges, and only $k\binom{s}{2}$ such pairs from the same block, the fraction of pairs of edges that cross in such a drawing is at most $1/k$. See Figure~\ref{fig:biarc_example}(i). Similarly, if the subsequence is block-nondecreasing of depth $k$ and block-size $s$, we draw the corresponding edges as monotone biarcs, consisting of two semicircles with the first (left) one above the spine, and the second (right) one below the spine. Furthermore, we draw the monotone biarc s.t. it crosses the spine at $b + 1 - \ell/n - r/(2n^2)$, where $\ell$ and $r$ are the left and right endpoints of the edge respectively. See Figure~\ref{fig:biarc_example}(ii).  By the same argument above, the fraction of pairs of edges that cross in such a drawing is at most $1/k$.

Hence, $E'$ can be decomposed into $C_k+(k-1)^2$ many monotone biarc diagrams, s.t. each monotone biarc diagram has at most $1/k$-fraction of pairs of edges that are crossing.

\begin{figure}[ht]
    \centering
    \includegraphics{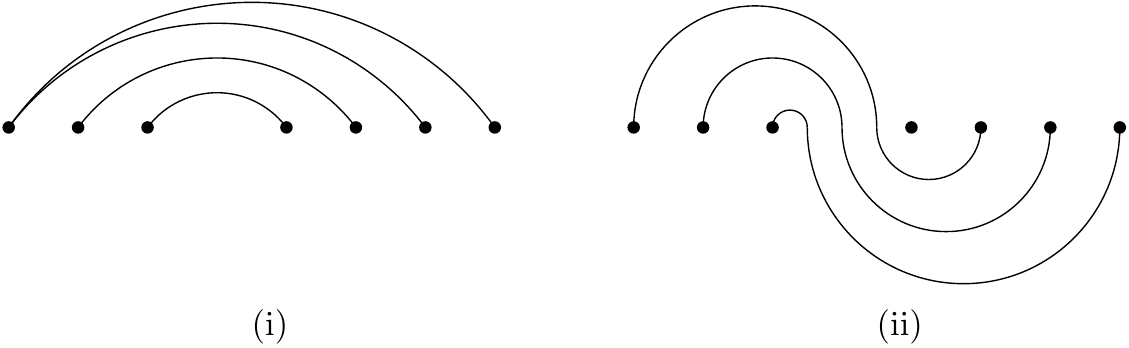}
    \caption{(i) a proper arc diagram. (ii) a monotone biarc diagram.}
    \label{fig:biarc_example}
\end{figure}

For edges within $[b]$, Lemma \ref{biarc_lemma} and the inductive hypothesis tell us that they can be decomposed into $(C_k+(k-1)^2)(\log|E|-1)$ monotone biarc diagrams, s.t. the fraction of pairs of edges that are crossing in each diagram is at most $1/k$. The same argument applies to the edges within $[n]\setminus[b]$. However, notice that two such monotone biarc diagrams, one in $[b]$ and another in $[n]\setminus[b]$, can be drawn on the same page without introducing more crossings. Hence, we can decompose $E\backslash E'$ into at most $(C_k+(k-1)^2)(\log|E|-1)$ such monotone biarc diagrams, giving us a total of $(C_k+(k-1)^2)\log|E|$ monotone biarc diagrams.\end{proof}

\section{Final remarks}\label{sec:remarks}

1. We call a sequence $(a_1,a_2,\dots,a_n)$ of $n$ distinct real numbers \emph{$\epsilon$-increasing} (resp. \emph{$\epsilon$-decreasing}) if the number of decreasing (resp. increasing) pairs $(a_i,a_j)$, where $i<j$, is less than $\epsilon n^2$. And we call a sequence \emph{$\epsilon$-monotone} if it's either $\epsilon$-increasing or $\epsilon$-decreasing. Clearly, a block-monotone sequence of depth $k$ is an $\epsilon$-monotone sequence with $\epsilon=k^{-1}$.  Hence, Theorem~\ref{main} implies the following.

\begin{corollary}\label{main_eps}
For all $n > 0$ and $\epsilon>0$, every sequence of $n$ distinct real numbers contains an $\epsilon$-monotone subsequence of length at least $\Omega(\epsilon n)$.
\end{corollary}

\noindent This corollary is also asymptotically best possible. To see this, for $n>(k-1)^2$ and a sequence $A=(a_i)_{i=1}^n$ of distinct real numbers, we can apply Corollary~\ref{main_eps} with $\epsilon=(64k)^{-1}$ to $A$ and obtain an $\epsilon$-monotone subsequence $S\subset A$ and then apply Lemma~2.1 in \cite{pach2021planar} to $S$ to obtain a block-monotone subsequence of depth $k$ and block-size $\Omega(n/k^2)$. So Corollary~\ref{main_eps} implies Theorem~\ref{main}.

\medskip

\noindent 2. Let $f(k)$ be the smallest number $N$ s.t. every finite sequence of distinct real numbers can be partitioned into at most $N$ block-monotone subsequences of depth at least $k$ upon deleting $(k-1)^2$ entries.  Our Theorem~\ref{partition} is equivalent to saying $f(k)=O(k^2\log k)$. The $K(k,2)$-type construction in Remark~\ref{extreme_construction_remark} implies $f(k)\geq k$. What is the asymptotic order of $f(k)$?

\medskip
\noindent 3. We suspect our algorithm presented in Theorem~\ref{main} can be improved. How fast can we compute a block-monotone subsequence as large as asserted in Theorem~\ref{main}? Can we do it within time almost linear in $n$ for all $k$?

\medskip
\noindent {\bf Acknowledgement.} We wish to thank the anonymous SoCG referees for their valuable suggestions.

\bibliographystyle{abbrv}
\bibliography{bib}
\end{document}